\documentclass{amsart}
\usepackage{graphicx}
\usepackage{amssymb,amscd,amsthm,amsxtra}
\usepackage{latexsym}
\usepackage{epsfig}
\usepackage{mathtools}
\usepackage{esint}
\usepackage{color}

\newtheorem{thm}{Theorem}[section]
\newtheorem{cor}[thm]{Corollary}
\newtheorem{lem}[thm]{Lemma}
\newtheorem{prop}[thm]{Proposition}
\theoremstyle{definition}
\newtheorem{defn}[thm]{Definition}
\theoremstyle{remark}
\newtheorem{rem}[thm]{Remark}
\numberwithin{equation}{section}

\newcommand{\R}{\mathbb R}

\newcommand{\eps}{\epsilon}

\newcommand{\p}{\partial}

\newcommand{\comment}[1]{}

\begin{document}

\title[]{Thin one-phase almost minimizers}
\author{D. De Silva}
\author{O. Savin}


\begin{abstract}We consider almost minimizer to the thin-one phase energy functional and we prove optimal regularity of the solution and partial regularity of the free boundary. We recover the theory for energy minimizers developed in \cite{DR, DS1}. Our methods are based on a noninfinitesimal notion of viscosity solutions we introduced in \cite{DS4}. \end{abstract}
\maketitle

\section{Introduction}

The purpose of this paper is the study of almost minimizers of the so-called {\it thin one-phase} energy functional, that is
 \begin{equation} \label{E} E(u,\Omega) := \int_\Omega |\nabla u|^2 dX + \, \mathcal{H}^n(\{(x,0) \in  \Omega : u(x,0)>0\}),\end{equation}
where $\Omega$ is a bounded domain in $\R^{n+1} = \R^n \times \R$ and points in $\R^{n+1}$ are denoted by $X=(x,x_{n+1})$. 

The minimization problem for $E$ was first considered by Caffarelli, Roquejoffre and Sire in \cite{CRS}, as a model of a Bernoulli type free boundary problem in the context of the fractional Laplacian. When $n=2$, this problem is related to models involving traveling wave solutions for planar cracks. In this setting the slit $\{u(x,0)=0\}$ represents the location of the crack in a 3D material and the free boundary $\p \{u(x,0)>0\}$ is one-dimensional and represents the edge of the crack. 

The study of the regularity of thin one-phase free boundaries was initiated by the first author and Roquejoffre in \cite{DR}, where it was shown that ``flat" free boundaries are $C^{1,\alpha}$ via a viscosity approach.  In \cite{DS1,DS2,DS3} we investigated further properties of minimizers by combining variational and nonvariational techniques. We showed that Lipschitz free boundaries are of class $C^\infty$ and local minimizers of $E$ have smooth free boundary except possibly for a small singular set of Hausdorff dimension $n-3$. Thus, the main regularity results for the classical one-phase Alt-Caffarelli energy functional have been extended to the thin setting \cite{AC, C1, C2, KNS}. 

In this paper we develop the regularity theory for almost minimizers of $E$ and their free boundaries. 
Almost minimizers of the Alt-Caffarelli functional were investigated recently by David and Toro, and David, Engelstein and Toro in \cite{DaT, DaET}. 
Our strategy differs from the one in \cite{DaT, DaET}. It is inspired by our recent work \cite{DS4} in which we develop a Harnack type inequality for functions that do not necessarily satisfy an infinitesimal equation but rather exhibit a two-scale behavior. As an application, we provide in \cite{DS4} a non-variational proof of of the $C^{1,\alpha}$ estimates of Almgren and Tamanini \cite{A, T} for quasi-minimizers of the perimeter functional. We follow here a similar approach, by showing that almost minimizers of $E$ are ``viscosity solutions" in this more general sense (see Subsection 3.5). Roughly, our viscosity solutions satisfy comparison in a neighborhood of a touching point whose size depends on the properties of the test functions.Once this is established, then we employ the techniques developed in \cite{DR, DS1, DS2} to study the regularity of the free boundary of viscosity solutions.

Before stating our main results, we recall the definition of almost minimizers (see \cite{G} for a comprehensive treatment of almost minimizers of regular functionals of the calculus of variations.)

\begin{defn}\label{almost_E} We say that $u$ is an almost minimizer for $E$ in $\Omega$ (with constant $\kappa$ and exponent $\beta$) if $u \in H^1(\Omega)$, $u \geq 0$ a.e. in $\Omega$, and
\begin{equation}\label{almost_min}
E(u, B_r(X)) \leq (1+ \kappa r^\beta) E(v, B_r(X))
\end{equation}
for every ball $B_r(X)$ such that $\overline{B_r(X)} \subset \Omega$ and every $v \in H^1(\Omega)$ such that $v=u$ on $\p B_r(X)$ in the trace sense.
\end{defn}

Our first theorem concerning the optimal $C^{0,1/2}$ regularity holds for a slightly weaker class than almost minimizers. For simplicity, we state it here for almost minimizers and we refer the reader to Theorem \ref{weak_main_thin} for the more general result.

\begin{thm}\label{intro_main_thin} Let $u$ be an almost minimizer for $E$ in $B_1$ with constant $\kappa$ and exponent $\beta$. Then
 $$\|u\|_{C^{0,1/2}(B_{1/2})} \le C (\|u\|_{H^1(B_1)} +1)$$ for some constant $C$ depending on $\kappa$, $\beta$ and $n$.
Moreover, $u$ has uniformly bounded $C^{0,1/2}$ semi-norm in a small ball $B_{r_0}$, 
with $r_0$ depending on $\kappa, \beta, n$ and $\|u\|_{H^1(B_1)}$.
\end{thm}

Furthermore we extend the main result in \cite{DR} concerning the regularity of the free boundary
$$F(u):=\p_{\R^n} \{x| \quad  u(x,0)>0\}, $$ to the context of almost minimizers. Precisely, we prove an improvement of flatness theorem (see Theorem \ref{flat_thin}), from which the following main regularity result follows.

\begin{thm}\label{thm_intro}
Let $u$ be an almost minimizer to $E$ in $B_1$ with constant $\kappa$ and exponent $\beta$. Then
$$\mathcal H^{n-1}(F(u) \cap B_{1/2}) \le C(\beta, \kappa, n),$$
and $F(u)$ is $C^{1,\alpha}$ regular outside a closed singular set of Hausdorff dimension $n-3$, for some $\alpha(\beta,n)>0$ small. 
\end{thm}

Our strategy also allows us to obtain $C^{1,\alpha}$ regularity of Lipschitz free boundaries via the arguments of \cite{DS1} (see Theorem \ref{thm2}).

The paper is organized as follows. In Section 2 we prove the optimal regularity for almost minimizers and non-degeneracy properties. The following section is devoted to the ``flatness implies $C^{1,\alpha}$" result. Section 4 deals with the linearized problem associated to our flatness theorem. Finally, in the last section, we collect the regularity results that follow by standard arguments on the basis of the theory developed in the previous sections.

\section{$C^{0,1/2}$ regularity and non-degeneracy of thin almost minimizers}

In this section we start our study of {\it thin almost minimizers}, that is almost minimizers to the thin one-phase energy functional defined in \eqref{E}. 

\smallskip

{\bf Notation.} Recall that $X=(x,x_{n+1}), x \in \R^n$ are points in $\R^{n+1}$. We denote by $B_r$ a $n+1$ dimensional ball of radius $r$, while we denote by $\mathcal B_r := B_r \cap \{x_{n+1}=0\}.$ Also, $B_r^+ := B_r \cap \{x_{n+1}>0\}.$

\smallskip

\subsection{$C^{0,1/2}$ regularity.}The purpose of this subsection is to obtain $C^{0,1/2}$ regularity of thin almost minimizers. In fact, we do not need to require that $u$ is an almost minimizer but we can weaken our assumption (see statement of Theorem \ref{weak_main_thin}.)

First we remark that the energy $E$ scales according to the $C^{0,1/2}$ rescaling
$$ u_\rho (X) := \rho^{-1/2} u(\rho X) \quad \Longrightarrow \quad E ( u_\rho, B_r)=\rho^{-n}E(u, B_{\rho r}).$$
If $u$ is an almost minimizer with constant $\kappa$ and exponent $\beta$, then $\tilde u$ is an almost minimizer with constant $\kappa \rho^\beta$ and exponent $\beta$. Thus, after an initial dilation we may assume that the constant $\kappa$ is sufficiently small.  

Our first result is the following dichotomy. From here on, constants depending only on $n$ are called universal, and they may change from line to line in the body of a proof. Recall also that the function $u$ is non-negative.

\begin{prop}\label{dich_thin} Let $u \in H^1(B_1)$ and assume that
\begin{equation}\label{first}E(u,B_1) \leq (1+\sigma) E(v,B_1)\end{equation}
for all $v \in H^1(B_1)$ such that $v=u$ on $\p B_1$ (in the trace sense.)
Denote by
\begin{equation}\label{a}a : = \left(\fint_{B_1} |\nabla u|^2 \ dX\right)^{1/2}.\end{equation}
There exist universal constants $\eta, M, \sigma_0>0$ such that if $\sigma \leq \sigma_0$  then the following dichotomy holds. Either
\begin{equation}
a \leq M
\end{equation}
or 
\begin{equation}\label{alt1}\left(\eta \fint_{B_{\eta}} |\nabla u|^2 \ dX\right)^{1/2} \leq \frac a 2. \end{equation}
\end{prop}\begin{proof} 
 Let $v$ denote the harmonic replacement of $u$ in $B_1.$ Then,
$$ \int_{B_1} |\nabla u -\nabla v|^2 \ dX \leq E(u, B_1) + \int_{B_1} (|\nabla v|^2 - 2 \nabla u \cdot \nabla v) \ dX$$
and by \eqref{first} together with the fact that 
$$\int_{B_1} \nabla v \cdot \nabla(u-v) \ dX=0$$ this gives ($C$ universal)
$$ \int_{B_1} |\nabla u -\nabla v|^2  \ dX\leq \sigma \int_{B_1} |\nabla v|^2 \ dX + C.$$ Thus, since $v$ minimizes the Dirichlet integral in $B_1$,
$$ \fint_{B_1} |\nabla u -\nabla v|^2 \ dX \leq \sigma \fint_{B_1} |\nabla u|^2 \ dX + C=\sigma a^2 +C.$$

Since $|\nabla v|^2$ is subharmonic in $B_1$ and $v$ minimizes the Dirichlet integral, we conclude that 
$$|\nabla v| \leq C_0 a \quad \text{in $B_{1/2}$},$$
with $C_0$ universal.
Thus, since $\nabla v$ is harmonic, if we denote by $q:=\nabla v(0)$, we conclude that $|q| \leq C_0a,$
and
$$\fint_{B_{\eta}} |\nabla v - q|^2 \ dX \leq C_1 a ^2\eta^2, \quad \quad \forall \, \, \eta \le 1/2,$$
with $C_1$ universal.
Thus, for $\bar C$ universal
\begin{equation}\label{b0}\fint_{B_{\eta}} |\nabla u - q|^2 \ dX \leq 2\sigma\eta^{-n-1}a^2 + 2 C_1 \eta^2 a^2 + \bar C \eta^{-n-1},\end{equation}
and we use $|q| \le C_0a$, hence 
\begin{equation}\label{bound1_thin}\eta \fint_{B_{\eta}} |\nabla u |^2 \ dX \leq 4\sigma\eta^{-n}a^2 + 4 C_1 \eta^3 a^2 + 2\bar C \eta^{-n} + 2 \eta C_0^2 a^2.\end{equation}
Now, we can choose $\eta$ small universal so that
$$2 \eta C_0^2 \leq \frac 1 8, \quad 4 C_1 \eta^3 \leq \frac{1}{24},$$
and $\sigma$ small so that
$$4 \sigma \eta^{-n} \leq \frac{1}{24}.$$
Then we distinguish two cases. Either
$$2\bar C \eta^{-n} \geq \frac{a^2}{24},$$
hence the first alternative holds, or $2\bar C \eta^{-n} < \frac{a^2}{24},$ which combined with \eqref{bound1_thin} and the choice of $\eta$ and $\sigma$ provides the bound in the second alternative.

\end{proof}

We can now prove our regularity theorem, which clearly implies Theorem \ref{intro_main_thin} in the introduction.

\begin{thm}\label{weak_main_thin}Let $u$ satisfy,
$$E(u, B_r(X)) \leq (1+\sigma) \, \,  E(v,B_r(X)) \quad \quad  \forall \, B_r(X) \subset B_1,$$and all competitors $v=u$ on $\p B_r(X).$ If $\sigma$ is small enough universal, 
then
 $$\|u\|_{C^{0,1/2}(B_{1/2})} \le C (\|u\|_{H^1(B_1)} +1)$$ for some constant $C$ universal. Moreover, $u$ has a uniform $C^{0,1/2}$ seminorm in a small ball $B_{r_0}$, 
with $r_0$ depending on $n$ and $\|u\|_{H^1(B_1)}$.
\end{thm}


\begin{proof} 
Let $\eta$, $M \ge 1$, and $\sigma_0$ be the constants from Proposition \ref{dich_thin} and assume $\sigma \leq \sigma_0$. Set
$$a(\tau):= \left(\tau\fint_{B_\tau} |\nabla u|^2 \ dX\right)^{1/2}.$$
We show that for all $k \ge 0$ the following inequality holds
\begin{equation}\label{ki}
a(\eta^k) \leq C(\eta) M + 2^{-k} a(1),
\end{equation}
with $C(\eta)$ a large constant.

For $k=0$ the desired inequality is clearly satisfied. Let us assume that it holds for $k$ and let us show that it holds also for $k+1.$

By Proposition \ref{dich_thin} (rescaled) either 
$$a(\eta^{k}) \leq M,$$ 
or 
$$a(\eta^{k+1}) \leq \frac 1 2 a(\eta^{k}).$$
Then \eqref{ki} holds also for $k+1,$ since if the first alternative holds, we can use that $$a(\eta^{k+1}) \leq C(\eta) a(\eta^{k}).$$ In conclusion $a(r) \le C (1+ a(1))$ for all $r <1$, and by Poincare inequality we conclude that
$$r^{-1/2} \fint_{B_r} (u - \bar u)\ dX \leq C(n, \|u\|_{H^1}), \quad \bar u= \fint_{B_r} u \ dX.$$
The same inequality can be obtained for the averages over all balls with center in $B_{1/2}$ which are included in $B_1,$ and this gives $$\|u\|_{C^{0,1/2}(B_{1/2})} \le C,$$ 
by Morrey-Campanato theory. 

\end{proof}

 From now on we discuss properties of almost minimizers near its zero set. By Theorem \ref{weak_main_thin} we may assume after a dilation that
$$ \|u\|_{C^{0,1/2}(B_1)} \le C, $$
for some $C$ depending only on $n$. 
Notice that by Caccioppoli inequality (Theorem 6.5 in \cite{G}), we also have that $$\|u\|_{H^1(B_{1/2})} \leq C'.$$ 
For completeness we sketch how this last bound is obtained, and we need much weaker hypotheses. If $u$ satisfies a rough energy inequality of the type 
$$E(u,B_r) \le 2 (E(v,B_r) + 1)  \quad \forall  \, r \in (1/2,1),$$
then, by taking $v$ an interpolation between $0$ in $B_r$ and $u$ in $B_{r+t} \subset B_1$, one obtains 
$$\int_{B_r}|\nabla u|^2 dX \le C \int_{B_{r+t}\setminus B_r} t^{-2} u^2 + |\nabla u|^2 dX + C.$$
Now by a standard iteration (see Lemma 6.1 in \cite{G}) it follows
$$\int_{B_{1/2}}|\nabla u|^2 dX \le C \int_{B_1} u^2 dX+ C .$$
In view of the Caccioppoli inequality we may assume after a dilation that also $E(u,B_1) \le C$. Then the energy inequality
\begin{equation}\label{en1}
E(u,B_1) \le (1+\sigma) E(v,B_1)
\end{equation} 
for any $v$ that equals with $u$ on $\p B_1$, implies
\begin{equation}\label{en2}
E(u,B_1) \le  E(v,B_1) + C'' \sigma,
\end{equation}
for some $C''$ depending only on $n$. It is more convenient working with \eqref{en2} instead of \eqref{en1} since the energies cancel in a region where $v=u$.

\subsection{Nondegeneracy}
After relabeling $\sigma$ in \eqref{en2} we assume that 
 $u$ satisfies
\begin{equation}\label{3000_thin}
\|u\|_{C^{0,1/2}(B_1)} \le C, \quad \quad \mbox{and} \quad E(u,B_1) \le  E(v, B_1) + \sigma,
\end{equation} 
with $\sigma$ small, for any $v \in H^1(B_1)$ which agrees with $u$ on $\p B_1$. Notice that by the standard Caccioppoli inequality $\|u\|_{H^1(B_{7/8})} \le C$.

\begin{rem}\label{rem01_thin} We remark that if $u$ vanishes at some point in $B_\rho$ then the rescaling $$ u_\rho(X):=\rho^{-1/2}u(\rho X)$$ satisfies \eqref{3000_thin} with $\sigma_\rho:= \rho^{-n} \sigma.$ 
\end{rem}  


We start with the following basic consequences of \eqref{3000_thin} by comparing $u$ with its harmonic replacement. Then we obtain a subsequent Harnack type inequality.

\begin{lem}[Harmonic replacement]\label{replace_thin} Assume \eqref{3000_thin} holds, $B_1 \subset \{u>0\}$, and let $v$ denote the harmonic replacement of $u$ in $B_{7/8}$. Then, 
\begin{equation}\label{uv}
\|u-v\|_{L^\infty(B_{1/2})} \leq c(\sigma), \quad c(\sigma) \to 0 \quad \text{as $\sigma \to 0$.}
\end{equation}
\end{lem}
\begin{proof}
By the maximum principle, $v>0$ in $B_1$, hence using \eqref{3000_thin} and the fact that $v$ is the harmonic replacement of $u$ we get 
$$\int_{B_{7/8}}|\nabla u -\nabla v|^2 dx \leq  \sigma.$$
By Poincare inequality we conclude that ($C$ changing from line to line)
$$\int_{B_{3/4}} (u-v)^2 dx \leq C\sigma$$
with $u-v$ uniformly $C^{0,1/2}$ in $B_{3/4}$. Thus, if $(u-v)(x) \geq \mu$ say at $x \in B_{1/2}$, we conclude that 
$$ \mu^{2n+4} \leq C \sigma.$$ Thus \eqref{uv} holds with $c(\sigma)= C \sigma^{1/(2n+4)}.$

\end{proof}

\begin{rem}\label{one_side}The upper bound in \ref{uv} holds also if we remove the assumption that $u$ is not strictly positive in $B_1$. Precisely,  if $v$ is a harmonic function in $D$, say $v$ continuous on $\p D$, with $v \geq u$ on $\p D$ and $D \subset \subset B_1$, we can conclude that $$u \leq v + c(\sigma,D') \quad \text{in $D' \subset \subset D$}, \quad c(\sigma,D') \to 0 \quad \text{as $\sigma \to 0$.}$$

Indeed, we can argue as in the proof of Lemma \ref{replace_thin}, using  $$V =\begin{cases}\bar v :=\min\{u,v\} \quad \text{in $D$}\\
u \quad \text{outside $D$,}
\end{cases}$$ as a competitor for the energy. We conclude that
$$E(u,D) -E(\bar v,D) \leq \sigma.$$
This leads to
$$\int_D |\nabla(u-\bar v)|^2 \ dX \leq \sigma,$$ and we can continue the argument as in Lemma \ref{replace_thin}, using that $v$ is $1/2$-H\"older continuous in $D'$.
\end{rem}

An immediate consequence of Lemma \ref{replace_thin} is the following version of Harnack inequality.
\begin{cor}\label{lhi_thin}
Assume \eqref{3000_thin} holds, $B_1 \subset \{u>0\}$, and let $w$ be a harmonic function such that $u \ge w$ in $B_1$ and $u - w \ge \mu$ at $0$ for some $\mu \le \mu_0$, $\mu_0$ small. Then $u-w \ge c \mu$ in $B_{1/2}$ for some $c$ universal provided that $\sigma \le \mu^{2n+5}$. 
\end{cor}

Our non-degeneracy lemmas can be stated as follows. Their proof follows the ideas in \cite{DS1}.

\begin{lem}[Weak non-degeneracy] \label{weaknd}Assume \eqref{3000_thin} holds, $B_1 \subset \{u>0\}$, and $\sigma$ is small. Then there exists a constant $c>0$ such that $u(0)\ge c.$
\end{lem}
\begin{proof}
Let $v$ be the harmonic replacement of $u$ in $B_{7/8}$ and then extended by $u$ outside $B_{7/8}$. Then according to Lemma \ref{replace_thin},  it is enough to prove the desired statement for $v.$

Now, let $\varphi \in C_0^{\infty}(B_{3/4})$ with $\varphi \equiv 1$ in $B_{1/2}$ and $0 \leq \varphi \leq 1.$ Since $v$ minimizes the Dirichlet integral and $v>0$ in $B_1$ (by the maximum principle), we have
\begin{equation}\label{Juv}E(v, B_1) \leq E(u, B_1) \leq E(v(1-\varphi), B_1) + \sigma.\end{equation}
On the other hand, since $v$ is harmonic, ($C$ universal possibly changing from line to line)
\begin{equation*}
\|v\|_{L^\infty(B_{3/4})}, \|\nabla v\|_{L^\infty(B_{3/4})} \leq C v(0),
\end{equation*}
from which we deduce that
\begin{equation}\label{cv}
\int_{B_1} |\nabla v|^2 dX \geq \int_{B_1} |\nabla v(1-\varphi)|^2 dX -C (v(0))^2.
\end{equation}
Combining \eqref{Juv}-\eqref{cv} (since $v>0$ in $B_1$) we get 
$$\int_{B_1}|\nabla v| ^2 + \mathcal H^n(\mathcal B_1) \leq Cv(0)^2 + \int_{B_1} |\nabla v|^2 + \mathcal H^n(\mathcal B_1 \setminus \mathcal B_{1/2}) +\sigma.$$
hence
$$C v(0)^2 \geq c_0 - \sigma \ge  c_0 /2,$$
for $c_0>0$ universal.

\end{proof}

\begin{lem}[Strong non-degeneracy]\label{snd}Let $u$ satisfy \begin{equation}\label{3001}
\|u\|_{C^{0,1/2}(B_1)} \le C, \quad \quad \mbox{and} \quad E(u,B_r(X)) \le (1+\sigma) E(v, B_r(X)) 
\end{equation} 
for any ball $B_r(X) \subset B_1$ with center $X \in \mathcal B_1$, and any $v \in H^1(B_r(X))$ which agrees with $u$ on $\p B_r(X)$.
If $0 \in F(u):=\p_{\R^n} \{u>0\}$, then $$\max_{B_r} u \ge c \, \, r^{1/2},$$ 
with $\sigma$, $c$ small positive constants depending only on $n$.
\end{lem}

\begin{proof} We recall that $\mathcal B_1$ denotes the $n$- dimensional ball included in $\{x_{n+1}=0\}$. By standard arguments, (see Lemma 7 in \cite{C3},) it suffices to show the  following claim. 

{\it Claim:} Let $X_0 \in \mathcal B_1 \cap \{u>0\}$ (close to the origin). There exists a sequence $X_k \in \mathcal B_1$ such that 
\begin{equation}\label{delta}
u(X_{k+1})=(1+ \delta) u(X_k)
 \end{equation}
with
\begin{equation}\label{delta2}
|X_{k+1}-X_k| \leq C dist(X, F(u))
 \end{equation}
for some $\delta$ small. 

We now show that the sequence of $X_k$'s  exists.  Assume we constructed $X_k$. After scaling we may suppose that $$u(X_k) = 1.$$ We call $Y_k$ the point where the distance from $X_k$ to $\{u=0\}$ is achieved. By the $C^{1/2}$ bound on $u$ and the non-degeneracy Lemma \ref{weaknd} we have $$c \leq d:=d(X_k)=|X_k-Y_k| \leq C.$$
Assume by contradiction that we cannot find $X_{k+1}$ in $\mathcal B_M$ with $M$ large to be specified later, with $$u(X_{k+1}) \geq 1+\delta.$$ Let $w$ be the harmonic function in $B_M^+$, such that $$w=0 \quad \text{on $\{x_{n+1} =0\} \cap \mathcal B_{M/2}$}, \quad w = CM^{1/2} \quad \text{on $\p B_M \cap \{x_{n+1}>0\}$},$$ and $w$ growing linearly in $\mathcal B_M \setminus \mathcal B_{M/2}.$ Here $C$ is chosen so that $u \leq w$ on $\p B_M \cap \{x_{n+1}>0\}$ (using the fact that $u$ is H\"older continuous and $u(0)=0$.)
Set $$v:=1+\delta +w.$$ Then $v$ satisfies the assumptions in Remark \ref{one_side}, with $D:= B_M^+$ and we can conclude that for $\sigma$ small (depending on $\delta$),
$$u \leq v + \delta \quad \text{on $B^+_{3M/4}$.}$$

On the other hand we have,
$$w \leq C x_{n+1} M^{-1/2} \leq \delta \quad \text{in $B^+:=B^+_{d}(X_k),$}$$ if $M$ is chosen large depending on $\delta.$ Thus, (by the same argument applied in $\{x_{n+1}<0\}$) $$\label{bound1}u \leq 1+3\delta \quad \text{in $B.$}$$
On the other hand, $u(Y_k)=0, Y_k \in \p B$. Thus from the H\"older continuity of $u$ we find $$\label{bound2}u \leq \frac 1 2, \quad \text{in $B_{c}(Y_k)$}.$$
We now use Lemma \ref{replace_thin} in $B$ (with $B_{d- c/2}$ instead of $B_{1/2}$) and conclude that if $\delta$ is sufficiently small we contradict that $$1=u(X_k)\leq \fint_{B_{d-\frac c 2}} u + c(\sigma).$$ 
\end{proof}

Once H\"older continuity and non-degeneracy have been established, it is straightforward to show that any blow-up sequence must converge uniformly on compact sets to a global minimizer. Moreover, by the Weiss monotonicity formula (see Theorem 4.1 in \cite{DS1}), it follows that the limit must be homogenous of degree 1/2. We sum up these facts in the next proposition.

\begin{prop}\label{blowup}
Assume $u$ is an almost minimizer of $E$ and that $0 \in F(u)$. Any blow up sequence converges uniformly (up to subsequences) to a global minimizing cone (homogenous of degree $1/2$) which has $0$ as a free boundary point. Also, their free boundaries converge in the Hausdorff distance (on compact sets) to the free boundary of the cone.
\end{prop}

\subsection{$C^{1,\gamma}$ estimates up to $\{x_{n+1}=0\}$}
 
Next we show that an almost minimizer $u$ in $B_1$
has $C^{1,\gamma}$ estimates up to the slit boundary in the case when $u$ vanishes on $x_{n+1}=0$. 

\begin{lem}\label{c1g}
Let $u$ be an almost minimizer to $E$ in $B_1$ (with constant $\kappa$ and exponent $\beta$). Assume $E(u,B_1) \le C_0$ universal, and $u=0$ on $x_{n+1}=0$. 
 Then, if $\kappa \le \kappa_0(\beta,n)$ small, we have
$$[u]_{C^{1,\gamma}(B^+_{1/2})} \le C(\beta), \quad \quad \gamma:=\beta/2.$$ 

\end{lem} 

\begin{proof}
We use competitors $v$ which also vanish on $x_{n+1}=0$, hence $u$ is an almost minimizer for the Dirichlet energy and then the proof is standard (see \cite{G}.). 

We sketch some details below. In any ball $B_r(X) \subset B_1$ included in $x_{n+1}>0$ or any half-ball with center on $x_{n+1}=0$, we have
\begin{equation}\label{2111}
\int_{B^+_r(X)}|\nabla (u - h)|^2 \le \kappa r^\beta \int_{B^+_r(X)}|\nabla u|^2,
\end{equation}
where $B_r^+(X)$ represents the restriction of $B_r(X)$ to the upper half space, and $h$ is the harmonic replacement of $u$ in $B_r^+(X)$. The desired conclusion follows by a standard Morrey-Campanato iteration: we show by induction that for a sequence of radii $r=\eta^k$, with $\eta$ small, depending on $\beta$ and $n$, we have
$$\fint_{B^+_r(X)}|\nabla u - q_{X,r}|^2 \le C_1 r^\beta, $$
for some vectors $q_{X,r}$ (which are in turn bounded by $C(C_1, \beta)$).

Indeed, by the $C^2$ estimates for $h$ we find from the induction hypothesis
$$|\nabla h - \nabla h(X)|^2 \le C(n) \eta^2 \, C_1 r^\beta \quad \mbox {in $B^+_{\eta r}(X)$,}$$
which combined with \eqref{2111} gives
$$ \fint_{B^+_{\eta r}(X)}|\nabla u - \nabla h(X)|^2 \le C(n) \eta^2 \, C_1 r^\beta + C(C_1, \beta) \,  \kappa \eta^{-(n+1)} r^\beta \le C_1(\eta r)^\beta,$$
provided that $\kappa$ is chosen sufficiently small.

\end{proof}

We give a rescaled version of the lemma above.

\begin{cor}\label{c1g2}
Assume that $\|u\|_{C^{1/2}(B_1)} \le C$, and $u$ is an almost minimizer to $E$ in $B_1$ with constant $\kappa \le \kappa_0$. Let $\mathcal B_r(x_0) \subset \mathcal B_{1/2}$ be a $n$-dimensional ball included in $\{u=0\}$ which is tangent to the free boundary. Then in $B_{r/2}^+(x_0)$ we have
$$[u]_{C^{1,\gamma} } \le C _1 r^{-(\frac 12 + \gamma)}.$$
\end{cor}

Indeed, assume $x_0=0$ for simplicity. Then the rescaled function
$$\tilde u (X):= r^{-1/2} u(rX),$$ 
is an almost minimizer with constant $\tilde \kappa=\kappa r^\beta \le \kappa_0$, and then the conclusion follows by rescaling back the $C^{1,\gamma}$ estimates for $\tilde u$.

Next we use this boundary $C^{1,\gamma}$ estimate to establish the optimal growth of $u$ away from a free boundary point.

\begin{lem}\label{l55}
Assume that $u$ is an almost minimizer to $E$ in $B_1$ (with constant $\kappa$ small and exponent $\beta$), $\|u\|_{C^{1/2}(B_1)} \le C$ universal, and $0 \in F(u)$. Then
$$u(X) \ge c \, \,  d_Z(X) \, |d_F(X)|^{-1/2} \quad \quad \mbox{in $B_{1/2}$,}$$
where $d_Z(X)$, $d_F(X)$ represent the distance form $X$ to the zero set of $u$, $Z:=\{u=0\} \cap \{x_{n+1}=0\}$ respectively the free boundary $F(u)$, and $c=c(n)>0$ is small.
\end{lem}

\begin{proof}We fix a point $X_0$, and we will prove the inequality at this point. Without loss of generality, after a rescaling, we may assume that $d_{F}(X_0)=|X_0|=1/2$, and the distance from $X_0$ to $F(u)$ is realized at the origin. Then
\begin{equation}\label{5111} 
\int _{B_1}|\nabla (u-h)|^2  \le \kappa,
\end{equation} 
where $h$ is  the harmonic replacement of $u$ in $B_1 \setminus Z$. By Poincare inequality and the fact that $u$, $h$ are uniformly H\"older away from the boundary of $B_1\setminus Z$, we find that
$$\|u-h\|_{L^\infty} \to 0 \quad \mbox{as $\kappa \to 0$, in the set $B_1 \setminus \{d_Z \le 1/10 \}$. }  $$
By Lemma \ref{snd}, $u$ is nondegenerate, which combined with the inequality above and the Harnack inequality for $h$, gives the conclusion of the lemma if $d_Z(X_0) \ge 1/10$.

If $d_Z(X_0) < 1/10$, then in $B_{1/2}(X_0) \cap \{x_{n+1}=0\}$ both $u$ and $h$ vanish, and by Lemma \ref{c1g} we have
that $u-h$ is uniformly $C^{1,\gamma}$ in $B_{1/4}(X_0)$, on both sides of $x_{n+1}=0$. This together with inequality \eqref{5111} implies that
$$\|u-h\|_{C^{0,1}(B_{1/4}(X_0))} \to 0 \quad \mbox{as $\kappa \to 0$.}$$
By the Hopf lemma for $h$ we find $u \ge c |x_{n+1}|=c \, \, d_Z$ provided that we choose $\kappa$ sufficiently small, and the desired conclusion follows. 

\end{proof}

\section{Regularity of flat thin free boundaries}

In this section we begin our investigation of the regularity of the free boundary for almost minimizers. After multiplying $u$ by a suitable constant we may assume that  $E$ is given by 
$$E(u,\Omega) := \int_\Omega |\nabla u|^2 dX + \, \frac \pi 2 \, \mathcal{H}^n(\{(x,0) \in  \Omega : u(x,0)>0\}).$$
The factor $\pi/2$ is chosen such that the function $U(t,s)$ which is the harmonic extension of $\sqrt{t^+}$ to the upper half-plane $\R^{2}_+=\{(t,s) \in \R \times \R, s>0\}$ and then reflected evenly across $\{s=0\}$ is a global minimizer for $E$. In other words $U$ is the real part of $\sqrt z$, or in the polar coordinates $$t= \rho \cos\theta, \quad  s=\rho \sin\theta, \quad \rho \geq 0, \quad -\pi \leq  \theta \leq \pi,$$ $U$ is given by
\begin{equation}\label{U}U(t,s) = \rho^{1/2}\cos \frac \theta 2. \end{equation}
When we extend $U$ constantly in $n-1$ variable we obtain a global minimizer in $\R^{n+1}$. Precisely $$U(X):= U(x_n, x_{n+1}),$$ denotes the {\it trivial cone} in $\R^{n+1}$.

One of our main results is that an almost minimizer with small constant $\kappa$ which is sufficiently close to an optimal configuration $U$ in $B_1$ has a 
$C^{1,\alpha}$ free boundary in $B_{1/2}$.

\begin{thm}[Flatness implies regularity.] \label{flat_thin} Let $u$ be an almost minimizer to $E$ in $B_1$ (with constant $\kappa$ and exponent $\beta$), and let $\|u\|_{C^{0,1/2}(B_1)} \leq C$.
Assume $|u-U| \le \tau_0$ in $B_1$. If $\tau_0$ and $\kappa$ are small enough depending on $n$ and $\beta$, then $F(u)$ is $C^{1,\alpha}$ in $\mathcal B_{1/2}$, with $\alpha = \alpha(\beta,n)>0$.
\end{thm}

\subsection{Translations of $U$}\label{s6.1} We start by normalizing our flatness assumption, that is 
we show that if $u$ is close to $U$ in $L^\infty$ then we can always trap it between two small translates of $U$. Precisely if 
\begin{equation}\label{600.0}
\|u-U\|_{L^\infty(B_2)} \le \tau, \quad \mbox{small},
\end{equation}
then it follows that 
\begin{equation}\label{600}
U(X-\eps e_n) \le u \le U(X+\eps e_n) \quad \mbox{in $B_1$}, 
\end{equation}
with $\eps=\eps(\tau) \to 0$ as $\tau \to 0$, provided that the constant $\kappa=\kappa (\tau)$ is sufficiently small.

The inequality \eqref{600} is a consequence of \eqref{600.0} away from a small neighborhood of the set $U=0$. We need to show the inequality is satisfied also in this neighborhood.

Fix $\eps>0$ small. The non-degeneracy property of $u$ implies that its free boundary lies in the strip $|x_n| \le \eps/2$.
Let $h$ be the harmonic replacement of $u$ in $B_{3/2}^+$. 
Since $u$ vanishes on $\mathcal B_{3/2} \cap \{x_n \le -\eps/2\}$, by Lemma \ref{c1g}, the $C^{1,\gamma}$ norm of $u$ is bounded in a small neighborhood $\mathcal U$ of $\mathcal B_1 \cap \{x_n \le - \eps\}$ by a constant $C(\eps)$. The same holds for $h$, and since  
$$\int_{B_{3/2}^+}|\nabla (u-h)|^2 \to 0 \quad \mbox{as $\kappa \to 0$,}$$
we find $$\|u-h\|_{C^{0,1}(\mathcal U)} \to 0.$$
Moreover, $\|h-U\|_{L^\infty} \to 0$ as $\tau \to 0$, hence by the regularity of harmonic functions 
$$\|U-h\|_{C^{0,1}(\mathcal U)} \to 0.$$
Combining these inequalities we find that the Lipschitz norm of $u-U$ tends to $0$  as $\tau, \kappa \to 0$, and this gives the desired inequality \eqref{600} in $\mathcal U$ as well.

\subsection{Reduction to the case when $u$ is even in $x_{n+1}$.}\label{s6.2} In order to apply directly the methods from \cite{DR} to the study of almost minimizers it is convenient to reduce our analysis to the case of functions that are symmetric with respect to the last variable. We remark however that this reduction is not really necessary, as the arguments in \cite{DR} can be extended to the nonsymmetric case without much difficulty.

 We show that we may replace $u$ by its even part with respect to $x_{n+1}=0$,
$$u_e(x,x_{n+1}):= \frac 12 \left( u(x,x_{n+1}) + u(x,-x_{n+1}\right ),$$
and the key properties are preserved.

Suppose that $E(u,1) \le C$ and $u$ satisfies (see \eqref{en1}-\eqref{en2})
\begin{equation}\label{601}
E(u,B_1) \le E(v,B_1) + C \sigma, \quad \mbox{whenever $v=u$ on $\p B_1$,}
\end{equation}
for some small constant $\sigma$. Then we claim that $u_e$ satisfies the same inequality with constant $\sigma^{1/3}$ instead of $C \sigma$.

Indeed, let $\bar u$ represent the harmonic replacement of $u$ in $B_1^+ \cup B_1^-$ which does not change the values of $u$ on $x_{n+1}=0$. Then, using \eqref{601} we find
$$\|\nabla (u-\bar u)\|^2_{L^2} \le C \sigma.$$
Clearly this inequality holds for the odd parts of $u$ and $\bar u$ as well. 

Let $u_o=u-u_e$ denote the odd part of $u$. Notice that $\bar u_o$ is an odd harmonic function in whole $B_1$. We write $u_e=u-\bar u_o + \bar u_o - u_o$ and use that 
$$ \|\nabla (u -\bar u_o)\|^2_{L^2} \le C , \quad \|\nabla (\bar u_o- u_o)\|^2_{L^2} \le C \sigma,$$
to obtain
$$E(u_e,B_1) \le E(u-\bar u_o,B_1) + C \sigma ^{1/2}.$$  
A function which equals $u_e$ on $\p B_1$ can be written as $v-\bar u_o$ with $v=u$ on $\p B_1$. Since 
$$\int_{B_1}  (|\nabla (u-\bar u_o)|^2 - |\nabla (v-\bar u_o)|^2) dX =\int_{B_1}  (|\nabla u|^2 - |\nabla v|^2) dX,$$
we find
$$E(u-\bar u_o, B_1) = E(v-\bar u_o) + E(u,B_1) - E(v,B_1) \le E(v-\bar u_o,B_1) + C \sigma,$$
and in the last inequality we used the hypothesis on $u$.
In conclusion $$E(u_e,B_1) \le  E(v-\bar u_0) + C \sigma^{1/2}$$
which proves our claim.

\subsection{Improvement of flatness lemma}

For the remaining of this section we assume that in $B_1$ the function $u$ satisfies

$(H1)$ $u$ is even in $x_{n+1}$, $\|u \|_{C^{0,{1/2}}} \le C,$

$(H2)$ $$E(u,B_1) \le  E(v,B_1) + \sigma,$$
for some $\sigma$ small, 

$(H3)$ in any $n$-dimensional ball $\mathcal B_r(X_0)$ included in $Z:=\{u=0\} \cap\{x_{n+1}=0\}$ which is tangent to the free boundary we have 
$$[u]_{C^{1,\gamma} } \le C r^{-(\frac 12 + \gamma)} \quad \quad  \mbox{in} \quad B_{r/2}^+(X_0).$$

$(H4)$ $$u \ge c \, d_Z d_F^{-1/2}$$ where $d_Z$ and $d_F$ represent the distances to the zero set $Z$, and  respectively the free boundary $F(u)$.

We remark that $(H3)$ and $(H4)$ hold for almost minimizers by Corollary \ref{c1g2} and Lemma \ref{l55}, and they remain valid after symmetrizing $u$ in the $x_{n+1}$ variable as in subsection \ref{s6.2}.

Theorem \ref{flat_thin} follows from the improvement of flatness lemma below.

\begin{lem}\label{limp}
Let $u$ satisfy $(H1)$-$(H4)$. Assume that $0 \in F(u)$ and 
 $$ U(X-\eps e_n) \le u \le U(X+\eps e_n).$$
Given $\alpha \in (0,1)$ there exists $\eta$ depending on $n$ and $\alpha$ such that if $\sigma \le \eps^{C}$ for some large constant $C=C(\beta,n)$, then 
\begin{equation}\label{flatimp2}U(x \cdot \nu  - \eps \eta^{1+\alpha}, x_{n+1}) \leq u(X) \leq U(x\cdot \nu+\eps \eta^{1+\alpha}, x_{n+1}) \quad \textrm{in $B_\eta$},\end{equation} for some direction $\nu \in \R^n, |\nu|=1,$ 
provided that $\eps \le \eps_0(\alpha, n, \beta)$ sufficiently small.
\end{lem}

We explain how Lemma \ref{limp} implies Theorem \ref{flat_thin}. Assume $0 \in F(u)$ and then it suffices to show by induction that for a sequence of radii $r=\eta^k$ we can find unit directions $\nu_r$ such that
$$ U(x \cdot \nu_r  - \eps_0 r^{1+\alpha}, x_{n+1}) \leq u_e(X) \leq U(x\cdot \nu_r+\eps_0 r^{1+\alpha},x_{n+1}) \quad \mbox{in $B_r$},$$
where $u_e$ denotes the even part of $u$ with respect to the last coordinate.

The case $k=0$ is guaranteed by subsection \ref{s6.1}. Assume the conclusion holds for some $r$, and say $\nu_r=e_n$. Then the rescaling $$\tilde u(X)=r^{-1/2} u(rX)$$ satisfies \eqref{en2} with $\sigma= \kappa r^\beta$ thus, by subsection \ref{s6.2}, its even part $\tilde u_e$ satisfies H1)-H4) above for $$\sigma=(\kappa r^\beta)^{1/3}.$$ On the other hand, by the induction step, $\tilde u_e$ satisfies the flatness hypothesis of Lemma \ref{limp} with $$\eps:=\eps_0 r^\alpha.$$ In order to apply Lemma \ref{limp} and obtain the desired conclusion in $B_{\eta r}$ we need to check $\sigma \le \eps^C$ which is satisfied if $\alpha$ and $\kappa$ are chosen sufficiently small.

The proof of Lemma \ref{limp} follows the lines of proof given for minimizers in \cite{DR}. We will complete its proof in the next section. In our setting the main step is to establish that $u$ is in an appropriate sense a viscosity solution to the thin one-phase problem. Then the remaining arguments carry through without much difficulty as in \cite{DR}.

\subsection{Comparison Subsolutions} Here we recall the definition and properties of a family of explicit subsolutions (and similarly supersolutions) introduced in \cite{DS2}. We later show that $u$ satisfies a version of the comparison principle with elements of this family.

We introduce the family $V_{\mathcal S, a}$ with $\mathcal S$ a smooth hypersurface in $\R^n$ and $a \in \R$. We remark that in \cite{DS2} we defined a slightly larger class $V_{\mathcal S,a,b}$ involving two real parameters $a$ and $b$ which was needed in order to establish the $C^{2,\alpha}$ estimates of viscosity solutions. For our purposes, it suffices to fix $b=0$.

 For any $a \in \R$ we define the following family of two-dimensional profiles given in polar coordinates $(\rho, \theta)$ in a plane of coordinates $(t,s)$:
\begin{equation}
v_{a}(t,s):= \left(1+\frac{a}{4}\rho \right) \, \, \rho^{1/2} \, \, \cos \frac{\theta}{2},
\end{equation}
that is 
 $$v_{a}(t,s) = \left(1+\frac{a}{4}\rho \right)U(t,s) = U(t,s) + o(\rho^{1/2}).$$

Given a $C^2$ surface $$\mathcal S= \{x_n = h(x')\} \subset \R^n,$$ and a point $X=(x,x_{n+1})$ in a small neighborhood of $\mathcal S$,  we call $\mathcal P_{\mathcal S,X}$ the 2D plane passing through $X$ and perpendicular to $\mathcal S$,
that is the plane containing $X$ and generated by the $x_{n+1}$-direction and the normal direction from $(x,0)$ to $\mathcal S$.

We define the family of functions 
\begin{equation}\label{vS}
V_{\mathcal S, a} (X): = v_{a}(t,s),
\end{equation}
with $t=\rho\cos \beta, s=\rho\sin \beta$ respectively the first and second coordinate of $X$ in the plane $\mathcal P_{\mathcal S,X}$. In other words, $t$ is the signed distance from $x$ to $\mathcal S$ (positive above $\mathcal S$ in the $x_n$-direction,) and $s= x_{n+1}$.

If $$\mathcal S:= \left \{ x_n = \xi'\cdot x' +\frac 1 2 (x')^T M x' \right \},$$ for some $\xi' \in \R^{n-1}$ and $M \in S^{(n-1) \times (n-1)}$  we use the notation \begin{equation}\label{vM}
 V_{M, \xi', a}(X):= V_{\mathcal S, a} (X). \end{equation}

For $\mu >0$ small, we define the following classes of functions $$\mathcal{V}_\mu: = \{V_{M, \xi',a} : \ 
\|M\|, |a|, |\xi'| \leq \mu\}.$$ 

Next lemma, which is Proposition 3.2 in \cite{DS2}, provides a condition for a function  $V \in \mathcal V_\mu$ to be a subsolution/supersolution.

\begin{lem}\label{l63} Let $V=V_{M,\xi',a} \in \mathcal V_\mu,$  with $\mu \leq \mu_0$ universal. There exists a universal constant $C_0>0$ such that if \begin{equation}\label{Vsub}a - tr M \geq C_0 \mu^2\end{equation} then  $V$ is a comparison subsolution to thin one-phase problem in $B_2$:
\begin{enumerate}
\item  $\Delta V \geq \mu^2 |x_{n+1}|$ in $B_2^+(V)$,

\item $\frac{\partial V}{\partial \sqrt t} =1$ on $F(V)$.
\end{enumerate} 

\end{lem}

The second property means that at any point $x_0 \in F(v)$ we have
$$V (x,z) = U((x-x_0) \cdot \nu(x_0), z)+ o(|(x-x_0,z)|^{1/2}), \quad \textrm{as $(x,z) \rightarrow (x_0,0),$}$$ where $\nu(x_0)$ denotes the unit normal at $x_0$ to $F(v)$ pointing toward $\mathcal{B}_2^+(V)$. 

We remark that property (i) is not stated precisely as above in Proposition 3.2 in \cite {DS2}, however it follows from its proof that $\triangle V \ge C \mu^2 |U_t| \ge \mu^2 |x_{n+1}|$.  

We think of the functions $V$ being obtained from $U$ through a domain deformation given by a map $X \mapsto X - \tilde V(X) e_n$, with $\tilde V(X)$ given implicitly as
$$ U(X)=V(X - \tilde V(X) e_n).$$
The choice of writing the perturbation in the $-e_n$ direction instead of the $e_n$ direction, is so that the transformation $V \rightarrow \tilde V$ preserves the ordering of functions: $V_1 \ge V_2 $ is equivalent to $\tilde V_1 \ge \tilde V_2$.  

We often work with translations of the graphs of the functions $V$. In terms of the hodograph transform $\tilde V$, a translation of the graph of $V$ by $-t_0 e_n$ corresponds to adding $t_0$ to $\tilde V$.

We recall Proposition 3.5 in \cite{DS2} which gives an estimates for $\tilde V$ up to order $O(\mu^2)$.

\begin{lem}\label{l64} Let $V=V_{M,\xi',a} \in \mathcal V_\mu$. Then $\tilde V$ satisfies the following estimate in $B_{2}$ 
\begin{equation*}
|\tilde V(X) - \gamma_V (X)| \leq C_1 \mu^2, \quad \gamma_V (X)= - \, \xi' \cdot x' -\frac 1 2 (x')^T M x' + \frac{a}{2}(x_n^2+x_{n+1}^2)
\end{equation*} with $C_1$ a universal constant.
\end{lem}

\subsection{Almost minimizers as viscosity solutions}\label{amvs}

The next lemma states that an almost minimizer cannot be above a a subsolution $V_{M,\xi',a}$ in $B_1$ and be {\it tangent} to it in $B_{1/4}$.  
\begin{lem}\label{l613}
Assume $u$ satisfies $(H1)-(H4)$, and $u \ge V=V_{M,\xi',a} \in \mathcal V_{\mu ^{2/3}}$ in $B_1$. If $$a - tr M \geq  \mu,$$
and $\sigma \le \mu^{C}$, $C=C(\gamma,n)$ large,
then  we have $$u(X) \ge V_{M,\xi',a}(X + t_0e_n) \quad \mbox{ in $B_{1/4}$ for some $t_0>0$}.$$
\end{lem}

A similar statement holds for supersolutions $V_{M,\xi',a}$ with $a - tr M \leq - \mu$ when the inequalities are reversed.

When we say that $V$ is tangent to $u$ by below in $B_{1/4}$  we mean that $u \ge V$ and any translation of $V$ by $t e_n$ with $t>0$, is strictly above $u$ at some point in $B_{1/4}$. 

Lemma \ref{l613} is used to provide a comparison principle for a function $u$ that satisfies $(H1)$-$(H4).$ 
Precisely, if $u$ is above a translation of $V(X-te_n)$ in $B_1$ and if in addition $u \geq V$ in say $B_1 \setminus B_{1/4}$ then this inequality can be extended to $u \geq V$ in $B_1.$

\begin{proof}
Let $$\bar V(X):=V_{\bar M, \xi', \bar a} \left(X + \frac {\mu}{8n} e_n \right ), \quad \bar a=a - \frac {\mu}{2 n}, \quad  M= M + \frac {\mu}{2n } I.$$

By Lemma \ref{l64} we can compare the hodograph transforms of $\bar V$ and $V$ and deduce
\begin{equation}\label{6020}
\bar V(X) \le V(X) \quad \mbox{ near $\p B_1$,} \quad \bar V(X) \ge V(X + \frac {\mu}{16 n}) \quad \mbox{in $B_{1/2}$.}
\end{equation}

Set $$u_{max}:=\max\{u, \bar V\}, u_{min}:=\min\{u, \bar V\}$$ and notice that 
$$u_{max}= u, \quad u_{min}=\bar V \quad \text{near $\p B_1$}.$$
Then we have
$$E(u, B_1)  \le E(u_{max}, B_1) +  \sigma,$$
or equivalently
\begin{equation}\label{602} E(u_{min}, B_1)-E(\bar V, B_1)  \le  \sigma.\end{equation} 
On the other hand we claim that 
\begin{equation}\label{603}
E(u_{min}, B_1)-E(\bar V, B_1) \geq \mu^{2} \int_{B_1}(\bar V-u_{min}) |x_{n+1}| \ dX.
\end{equation}
This inequality is a consequence of the fact that $\bar V$ is the minimizer of the energy
$$E(v,B_1) +  \mu^{2} \int_{B_1} v \, \,  |x_{n+1}| \ dX,$$
among all $H^1(B_1)$ functions $v$ that satisfy $\min\{V, \bar V\} \le v \le \bar V$.

Indeed, as it was shown in \cite{DS2}, minimizers do enjoy the comparison principle with comparison subsolutions in the unconstrained region $v < \bar V$. 
By Lemma \ref{l63} we know that in $B_1$ the translations of $\bar V$, $V_t:=V_{\bar M, \xi',\bar a} (X+ t e_n)$ with $t \in [0, \mu/8n]$ form a continuous family of such subsolutions for the minimization problem above since 
$$\bar a - tr \, \bar M \ge \mu/2 \quad \Longrightarrow \quad \triangle V_t \ge  \mu^2|x_{n+1}|.$$
Now the claim \eqref{603} follows since $V_0 \le \min\{V, \bar V\}$ and $V_{\mu/8n}=\bar V$. 

Let's assume by contradiction that $V$ is {\it tangent} by below to $u$ at some point $X_0 \in \bar B_{1/4}$, in the sense that any translation $V(x+te_n)$ with $t>0$ cannot be below $u$ in a neighborhood of $X_0$. In view of \eqref{602}-\eqref{603} it remains to show that in this case we satisfy a rough integral bound
\begin{equation}\label{604}
\int_{B_1} (\bar V - u_{min}) \, |x_{n+1}| dX \ge  \mu^{ C(\gamma)}.
\end{equation}

If $X_0$ is outside a $\mu^2$ neighborhood of the set $\bar V =0$ then, by using \eqref{6020}, we find
$$(\bar V - u) (X_0) \ge c \,  \mu \max V_{x_n}(X_0 + \xi e_n) \ge c \mu |X_0 \cdot e_{n+1}| \ge \mu^4.$$
The uniform $C^{0,1/2}$ bound of $u$ implies
$$\bar V - u_{min} \ge \bar V - u  \ge \frac 12 \mu^4 \quad \mbox{in} \quad B_{c\mu^8}(X_0),$$
 which gives the desired integral bound.

We consider the case when $X_0$ is in a $\mu^2$ neighborhood of the set $\bar V =0$. Let $r$ denote the distance from $X_0$ to $F(u)$ and then $$r \ge \mu^{3/2}.$$ Otherwise, by $(H4)$, we have $$u \ge c r^{-1/2} |x_{n+1}| \ge c \mu ^{-3/4} |x_{n+1}| \gg C \mu ^{-1} |x_{n+1}| \ge V(X+ \frac {\mu}{32} e_n),$$ in $B_{2r}(X_0)$ and we contradict that $V$ is tangent to $u$ by below at $X_0$. This argument and H4) show that $u=0$ in $B_{r}(X_0) \cap \{x_{n+1}=0\},$ hence by $(H3)$
$$\|u\|_{C^{1,\gamma}} \le C r^{-1/2 -\gamma} \le \frac 12 \mu^{-3} \quad \mbox{in} \quad B^+_{r/4}(X_0).$$ 
Since $V$ is tangent by below at $X_0$ and satisfies the same $C^{1,\gamma}$ bound, we get
$$|u-V| \le \mu^{-3} d^{1+ \gamma} \quad \mbox{in $B_d(X_0)$, $d \le r/4$}.$$
On the other hand, by \eqref{6020},
$$\bar V - u  \ge (V(X+ \frac {\mu}{16\, n}e_n) - V(X)) - (u-V),$$
and we use that the difference between the two translates of $V$ is greater than $c \mu\,  \partial _{x_n} V \ge c \mu |x_{n+1}|$. We obtain 
$$ \bar V -u \ge c \mu d - \mu^{-3} d^{1+ \gamma} \ge \frac c 2 \mu \, d \quad \mbox{in a ball $B_{cd} \subset B_d(X_0)$},$$
provided we choose $d= \mu ^{C(\gamma)}$ with $C(\gamma)$ sufficiently large, which implies the desired bound \eqref{604}.

\end{proof}

For completeness, we state also the case of supersolutions.

\begin{lem}\label{l61}
Assume $u$ satisfies $(H1)$-$(H4)$ and $u \le V=V_{M,\xi',a} \in \mathcal V_{\mu ^{2/3}}$ in $B_1$. If $$a - tr M \leq - \mu,$$
and $\sigma \le \mu^{C}$, $C=C(\gamma,n)$ large,
then  we have $$u(X) \le V_{M,\xi',a}(X - t_0e_n) \quad \mbox{ in $B_{1/4}$ for some $t_0>0$}.$$
\end{lem}

\begin{proof}
The proof is similar, and we only sketch the last part of argument which is slightly different. The rough integral bound which we need to prove now is
$$\int_{B_{1/2}} (\max\{u, \underbar V\} - \underbar V ) |x_{n+1}| \, dX \ge \mu^C, \quad \quad \underbar V:=V_{M,\xi',a}(X-\frac{\mu}{16 n}e_n),$$
if we assume that $V$ is tangent by above to $u$ at some point $X_0 \in B_{1/4}$. 

If $X_0$ is outside a $\mu^2$ neighborhood of the set $V =0$ then
$$(u - \underbar V)(X_0) \ge \mu^4,$$
and the conclusion follows from the Holder continuity of $u$ and $V$. 

When $X_0$ is in a $\mu^2$ neighborhood of the set $V =0$, let $r$ denote the distance from $X_0$ to $F(u)$. 
If $$r \le \mu ^{3/2}$$ then, by H4) we have
$$u \ge c \mu^{-3/2} |x_{n+1}| \ge \frac c 2  \mu^{-3/2} |x_{n+1}| + \underbar V, \quad \mbox{in} \quad B_{\mu^{3/2}(X_0)},$$ 
and the integral bound follows.
If $$r \ge \mu^{3/2},$$
then we may assume that $u=0$ in $B_r(X_0) \cap \{x_{n+1}=0\}$ since otherwise the conclusion is obvious in view of $(H4)$. We can use $(H3)$ and obtain that 
$$u \le C r^{-1/2}|x_{n+1}| \quad \mbox{in} \quad B_{r/2}(X_0).$$
This means that also $V=0$ on $B_{cr}(X_0) \cap \{x_{n+1}=0\}$ for some small $c$, otherwise we contradict that $V$ and $u$ are tangent at $X_0$. Using that in $B_{cr}(X_0)$ both $u$ and $V$ have $C^{1,\gamma}$ norm bounded by $C r^{-1/2 - \gamma} \ll \mu^{-3}$ we find
$$|u-V| \le \mu^{-3} d^{1+ \gamma} \quad \mbox{in $B_d(X_0)$, $d \le cr$},$$
and the last argument of the previous lemma applies as we write
$$u - \underline V \ge V- \underline V - |u-V|.$$ 
\end{proof}

\section{Proof of the Improvement of flatness Lemma.}

\subsection{The normalized Hodograph transform}
We follow the arguments in \cite{DR} and check that they apply in our situation. 
Here and henceforth we denote by $P$ the half-hyperplane $$P:= \{X \in \R^{n+1} : x_n \leq 0, x_{n+1}=0\}$$ and by $$L:= \{X \in \R^{n+1}: x_n=0, x_{n+1}=0\}.$$ 
 Also we denote by $U_b$ the translation of $U$ by $-b e_n$, 
$$U_b(X):=U(X + b e_n), \quad \quad b \in \R.$$

Assume that $u$ satisfies the hypotheses of Lemma \ref{limp} and therefore satisfies the $\eps$-flatness assumption \begin{equation}\label{flattilde}U(X - \eps e_n) \leq u(X) \leq U(X+\eps e_n) \quad \textrm{in $B_1.$}\end{equation}
We define the multivalued map $\tilde u$ as the $\eps$-normalized Hodograph transform of $u$ with respect to $U$ which associate to each $X \in B_{1-\eps} \setminus P$ the set $\tilde u(X) \subset \R$ via the formula 
\begin{equation}\label{deftilde} U(X) = u(X - \eps \tilde u(X) e_n).\end{equation} 
Since $U$ is increasing in the $e_n$ direction, we obtain that 
$$|\tilde u| \le 1 \quad \quad \mbox{in $B_{1-\eps} \setminus P$}.$$
The free boundary problem for $u$ is encoded in the limiting values of $ \eps \tilde u$ on $L$. 

\subsection{The linearized problem}
As in \cite{DR}, the strategy to prove Lemma \ref{limp} is to show that for small $\eps$, $\tilde u$ is well approximated uniformly on compact sets of $B_1$ by a viscosity solution to the associated linearized equation
\begin{equation}\label{linear}\begin{cases} \Delta (U_n h) = 0, \quad \text{in $B_1 \setminus P,$}\\ |\nabla_r h|=0, \quad \text{on $ L$.}\end{cases}\end{equation}

We recall the definition of a viscosity solution $h$ to the linearized problem above.

\begin{defn}\label{linearsol}We say that $h$ is a solution to \eqref{linear}  if $h \in C(B_1)$, $h$ is even in $x_{n+1}$ and it satisfies
\begin{enumerate}\item $\Delta (U_n h) = 0$ \quad in $B_1 \setminus P$;\\ \item $h$ cannot be touched by below (resp. by above) at any  $X_0=(x'_0,0,0) \in L$, by a continuous function $\phi$ which satisfies $$\phi(X) = \phi(X_0) + a(X_0)\cdot (x' - x'_0) + b(X_0) r + O(|x'-x'_0|^2 + r^{3/2}), $$ with $b(X_0) >0$ (resp. $b(X_0) <0$) and $r$ represents the distance from $X$ to $L$, $r^2=x_n^2+x_{n+1}^2$. \end{enumerate}\end{defn}

Lemma \ref{limp} follows easily from the $C^{1,\alpha}$ estimate for $h$ obtained in \cite{DR}, once the uniform convergence of $\tilde u$ to a solution $h$ is established. 
\begin{thm}\cite{DR}\label{classnewcor} Let  $h$ be a solution to \eqref{linear} such that $|h|  \leq 1$. Given any $\alpha \in (0,1)$, there exists $\eta$ depending on $\alpha$, such that $h$ satisfies
\begin{equation*}|h(X) - (h(0) +  \xi' \cdot x') | \leq \frac 1 4 \eta^{1+\alpha} \quad \textrm{in $B_{\eta},$}\end{equation*} 
for some vector $\xi' \in \R^{n-1}$.\end{thm}

\subsection{Two properties}
Next we state two properties (P1) and (P2) for the function $u$ which turn out to be sufficient for obtaining the approximation of $\tilde u$ with solutions $h$ of \eqref{linear}, and for obtaining the improvement of flatness Lemma \ref{limp}. These properties are written in terms of a small parameter $\delta>0$. 

\

(P1) {\it Harnack inequality} 

Given $\delta>0$, there exists $\eps_0=\eps_0(\delta)$ such that 
if $\eps \le \eps_0$ and $$u \ge U_b \quad \mbox{ in $B_r(X_0) \subset B_1$, and $|b| \le \eps$, $r \ge \delta$,}$$and 
$$u(Y) \ge U_{b + \tau \eps}(Y) \quad \mbox{ for some $Y$ with  $B_{r/4}(Y) \subset B_r(X_0) \setminus P$, and $ \tau \in [\delta,1]$,}$$ 
then $$u \ge U_{b + c \tau \eps} \quad \mbox{ in $B_{r/2}(X_0)$},$$ for some $c=c(n)>0$ universal.

Similarly, the above holds when we replace $\ge$ by $\le$ and $\tau$ by $-\tau$.

\ 

(P2) {\it Viscosity property}

Given $\delta>0$, there exists $\eps_0=\eps_0(\delta)$ such that if $\eps \leq \eps_0$ then

a)  we cannot have $u(X_0)=q(X_0)$ and $u \ge q$ in $B_{\delta}(X_0) \subset B_1 \setminus P$ where 
$$q \in C^2(B_\delta(X_0)) \quad \mbox{ such that } \quad \|D^2 q\| \le \delta^{-1},  \quad \triangle q \ge \delta \eps.$$

b) we cannot have $u \ge V$ in $B_\delta(X_0) \subset B_1$ with $X_0 \in L$ and $V$ tangent by below to $u$ in $B_{\delta/4}(X_0)$ where
$V$ is a translation of a function $V_{M,\xi',a} \in \mathcal V_{\delta^{-1}\eps}$, 
$$V(X):=V_{M,\xi',a}(X+t e_n), \quad \mbox{with} \quad a- tr \, M \ge \eps.$$

Similarly, a), b) hold when we compare $u$ with functions $q$, $V$ by above and $\triangle q \le - \delta \eps$ respectively $a- tr \, M \le - \eps$.

\medskip

When we say that $V$ tangent by below to $u$ in $B_{\delta/4}(X_0)$ we mean that any translation $V(x+se_n)$ with $s>0$ cannot be below $u$ in
$B_{\delta/4}(X_0)$.

\

We explain why (P1) and (P2) suffice for the proof of Lemma \ref{limp}. We argue by compactness and consider a sequence of functions $u_k$ satisfying the assumptions of the Lemma \ref{limp} with $\eps_k, \sigma_k \to 0$. We show that we can extract a subsequence of the $\tilde u_k$'s which converges on compact sets of $B_1$ to a solution $h$ of the linearized problem. Then Lemma \ref{limp} becomes a consequence of Theorem \ref{classnewcor}. 

We may suppose that $\eps_k \leq \eps_0(2^{-k})$ with $\eps_0(\delta)$ as in the properties (P1), (P2) above. Then property (P1) guarantees that $\tilde u_k$ satisfies the Harnack inequality in balls of size $r$ included in $B_1$ from scale $r=1$ up to scale $r = 2^{-k} $. By Arzela Ascoli theorem, the multivalued graph of $\tilde u_k$ converges (up to a subsequence) in the Hausdorff distance to the graph of a uniformly H\"older function $h$ defined in $B_1$.

Property (P2) implies that the limit function $h$ satisfies the linearized problem \eqref{linear}. Indeed, suppose that $Q$ is a quadratic polynomial that touches $U_n h$ strictly by below at a point $X_0 \in B_1 \setminus P$ with $\triangle Q >0$. By the uniform convergence of $\tilde u_k$ to $h$ it follows that a translation of $q_k$ defined as $$U(X)=:q_k(X-\eps_k \frac{Q}{U_n}  e_n),$$ touches $u_k$ by below at a point $X_k \to X_0$ and it is below $u_k$ in a fixed neighborhood of $X_0$.  This contradicts property (P1) part a) for large $k$ since it is straightforward to verify that (see Proposition 2.8 in \cite{DS2}) 
$$q_k \in C^2(B_\delta(X_k)), \quad  \|D^2 q_k\| \le \delta^{-1},  \quad \triangle q_k = \eps_k \triangle Q + O(\eps_k^2) \ge \delta \, \, \eps_k,$$
for a fixed $\delta>0$ depending on $Q$ and $X_0$.

Next we check that $h$ satisfies the correct boundary condition on $L$. 
We argue by contradiction.  Assume for simplicity (after a translation) that there exists a function $\phi$ which touches $h$ by below at $0$ with  $\phi(0)=0$ and such that $$\phi(X) = \xi' \cdot x'  + \beta  r + O(|x'|^2 + r^{3/2}), \quad \mbox{with} \quad \beta >0.$$   

Then we can find a constant $a \ge 1$ large (depending on $\phi$) such that the quadratic polynomial
$$Q(X)= \xi' \cdot x'  - \frac{a}{2}|x'|^2 +  n a  \, r^2 $$
 touches $h$ strictly by below at $0$ in $B_{2 \delta}$ for some sufficiently small $\delta$. We let
 $$V:=V_{\eps a I, -\eps \xi', 2n \eps a} \in \mathcal V_{ \delta^{-1} \eps},$$
 which satisfies the conditions of property (P2) part b).
 By Lemma \ref{l64} we know that the $\eps$-rescaling of the hodograph transform of $V$ satisfies
 $$\tilde V= Q + O(\eps).$$
 
 Since $\tilde u_k$ converges uniformly to $h$, we obtain that a translation of $V$ (with $\eps=\eps_k$) touches $u_k$ by below at some point $X_k \to 0$ and $u_k$ is above this translation in $B_\delta$. This contradicts property (P2) part b).

\subsection{Properties (P1) and (P2) are satisfied.}

We fix $ \delta>0$ and consider a function $u$ that satisfies the hypotheses of Lemma \ref{limp}, and recall that $\sigma \le \eps^C$ with $C=C(\beta,n)$ sufficiently large. We need to show that (P1) and (P2) hold if $\eps \le \eps_0(\delta)$ sufficiently small.

We distinguish 3 cases depending on the position of $B_r(X_0)$ with respect to $P$ and $L$.

\medskip

{\it Case 1:} $B_r(X_0) \subset B_1 \setminus P$. 

Then $\{ u>0 \}$ in $B_{7r/8}(X_0)$ and in this ball we replace $u$ by its harmonic replacement $v$.
The almost minimizer hypothesis H2) implies
\begin{equation}\label{680}
\int_{B_{7r/8}(X_0)}|\nabla(u-v)|^2 dX \le  \sigma.
\end{equation}
From the Poincare inequality and the uniform Holder continuity of $u$ and $v$ in $B_{3r/4}(X_0)$ we conclude
$$|u-v| \le  (C r^2 \sigma)^\frac{1}{2(n+2)}\le \eps^2 \quad \mbox{in} \quad B_{3r/4}(X_0).$$
Hence in $B_{3r/4}(X_0)$, $u$ is approximated by a harmonic function up to an $\eps^2$ error, and then properties (P1) and (P2) a) easily follow for this case provided that $\eps$ is sufficiently small.

\medskip

{\it Case 2:} $X_0 \in P$ and $B_r(X_0) \subset B_1 \setminus L$. 

We let $v$ denote the harmonic replacement of $u$ in $B_r^+(X_0)$. By $(H3)$ both $u$ and $v$ have $C^{1,\gamma}$ norm in $B^+_{7r/8}(X_0)$ bounded by $C r^{-(1/2 + \gamma)}$. Now the corresponding inequality \eqref{680} gives constants $K$, $k$ depending on $n$ and $\gamma$ such that
$$|\nabla (u-v)| \le C r^{-K} \sigma ^k \le \eps^2 \quad \mbox{in} \quad B^+_{3r/4}(X_0).$$
Then $u$ is approximated by a harmonic function up to an $\eps^2|x_{n+1}|$ error, and then property (P1) follows from the boundary Harnack inequality for harmonic functions.

\medskip

{\it Case 3:} $X_0 \in L$.

After a rescaling we may assume that $X_0=0$, $r=1$ and $\sigma$ is replaced by $C(\delta) \sigma$ (see Remark \ref{rem01_thin}). 

For property (P1), let's assume for simplicity that $b=0$, i.e. $u \ge U$ in $B_1$. By the first two cases above we already know the desired inequality holds (for some small constant $c(n)$) outside a small neighborhood of $L$. It remains to show that it holds (for a possibly smaller constant $c'(n)$) in a small neighborhood of $L$ and for this it suffices to do it in a neighborhood of $0$. We express the inequalities in terms of $\tilde u$, the $\eps$-rescaling of the Hodograph transform of $u$. Let $Q$ denote the quadratic polynomial
$$Q(X):= -\frac 12 |x'|^2  + n (x_n^2+ x^2_{n+1}),$$
and let $$V:=V_{\eps I, 0, 2n \eps}, \quad \quad \tilde V= Q + O(\eps),$$
denote the corresponding subsolution associated to $Q$.
We know $\tilde u \ge 0$ in $B_{3/4}$, and
$$  \tilde u \ge c_0(k_0,n)\, \,  \tau \quad \mbox{outside the cylinder} \quad \mathcal C :=\{ |x'| \le 3/4, \quad |(x_n,x_{n+1})| \le k_0 \}.$$
We choose $k_0(n)$ small, so that the two inequalities on $\tilde u$ above imply 
$$ \tilde u \ge \frac{c_0}{2n} \, \tau  \, (2^{-10} + Q)  \quad \mbox{in $B_{3/4} \setminus B_{1/16}$.}$$

In view of the comparison Lemma \ref{l61} applied to $u$ and translates of $V$, (with $\mu = \eps$) we conclude that 
$$ u \ge V(X+  c_1 \, \tau \, \eps \, e_n) \quad \mbox{in $B_{3/4}$}, \quad c_1:=\ 2 ^{-12} \frac{c_0}{n}.$$
This gives the desired inequality $u \ge U_{(c_1/2) \tau \eps}$ in a small neighborhood of $0$, and property (P1) is proved.

Property (P2) part b) is a direct consequence of Lemma \ref{l61}. Indeed, assume that $X_0=0$ and after a rescaling 
$$\bar u (X) = \delta^{-1/2} u(\delta X)$$
we may assume $\delta=1$, $\bar a = \delta a$, $\bar M= \delta M$, $\bar \xi= \xi$, $\bar \sigma = C(\delta) \sigma$. Then $\bar V \in \mathcal V_{\delta^{-1} \eps}$ and we can apply Lemma \ref{l61} with $\mu = \delta \eps$.

\section{Partial regularity of the free boundary} The blow-up convergence of almost minimizers to minimizing cones (see Proposition \ref{blowup}) and the flatness theorem Theorem \ref{flat_thin} imply that the blow up analysis of minimizing cones and minimizers performed in \cite{DS1} carries through identically to the the case of almost minimizers. In particular we recover the same regularity results as for minimizers up to $C^{1,\alpha}$ regularity. Recall that the free boundary of a minimizer consists of a singular part which is a closed set of Hausdorff dimension $n-3$, and a regular part which has finite $n-1$ dimension and is locally smooth, see \cite{DS1}. In our case this result can be written as follows.

\begin{thm}\label{thm_fin}
Let $u$ be an almost minimizer to $E$ in $B_1$ with exponent $\beta$ and sufficiently small constant $\kappa (\beta)$. Then
$$\mathcal H^{n-1}(F(u) \cap B_{1/2}) \le C, \quad \quad \mbox{with $C$ universal,}$$
and $F(u)$ is $C^{1,\alpha}$ regular outside a closed singular set of Hausdorff dimension $n-3$, for some $\alpha(\beta,n)>0$ small. 
\end{thm}

We also state the result of \cite{DS1} about the regularity of Lipschitz free boundaries for the case of almost minimizers.

\begin{thm}\label{thm2} Let $u$ be an almost minimizer in $B_1$ with exponent $\beta$ and constant $\kappa$. Assume that $0 \in F(u)$ and that $F(u)$ is a Lipschitz graph in the $e_n$ direction  with Lipschitz constant $L.$ Then $F(u) \cap B_{1/2}$ is a $C^{1,\alpha}$ graph, and its $C^{1,\alpha}$ norm is bounded by a constant that depends only on $n,L$ and $\beta$ and $\kappa$.\end{thm}

\end{document}